\documentclass{amsart}

\usepackage{amsmath, amssymb, mathrsfs}
\usepackage{xypic}
\newcommand{\mathsym}[1]{{}}

\newtheorem{thm}{Theorem}[section]
\newtheorem{lemma}[thm]{Lemma}
\newtheorem{prop}[thm]{Proposition}
\newtheorem{crl}[thm]{Corollary}

\theoremstyle{definition}
\newtheorem{exm}[thm]{Example}

\newtheorem{rem}[thm]{Remark}

\begin{document}

\title
 {positivstellens\"atze for real function algebras}
 \date{\today}

\author{Murray Marshall}
\author{Tim Netzer}
\address{Department of Mathematics and Statistics,
University of Saskatchewan,
Saskatoon, \newline \indent
SK, S7N 5E6, Canada}
\email{marshall@math.usask.ca}
\address{Fakult\"at f\"ur Mathematik und Informatik, Universit\"at Leipzig, PF 10 09 20 D-04009 \newline \indent Leipzig, Germany}
\email{Tim.Netzer@math.uni-leipzig.de}


\begin{abstract} We look for algebraic certificates of positivity for
functions which are not necessarily polynomial functions. Similar questions were examined earlier by Lasserre and Putinar and by Putinar in \cite[Proposition 1]{LP} and \cite[Theorem 2.1]{P}.
We explain how these results can be understood as results on \textit{hidden positivity}: The required positivity of the functions implies their positivity when considered as polynomials on the real variety of the respective algebra of functions.  This variety is however not directly visible in general.

We show how algebras and quadratic modules with this hidden positivity property can be constructed.
We can then use known results, for example Jacobi's representation theorem \cite[Theorem 4]{J}, or the Krivine-Stengle Positivstellensatz \cite[page 25]{M}, to obtain certificates of positivity
relative to a  quadratic module of an algebra of real-valued functions.
Our results go beyond the results of Lasserre and Putinar, for example when dealing with non-continuous functions. The conditions are also easier to check.

We explain the application of our result to various sorts
of real finitely generated algebras of semialgebraic functions. The emphasis is on the case where the quadratic module is also finitely generated. Our results also have application to optimization of real-valued functions, using the semidefinite programming relaxation methods pioneered by Lasserre \cite{L, L1,LP,M}.
\end{abstract}

\maketitle

\section{ Introduction }
 A \textit{Positivstellensatz} is a theorem that relates positivity of certain functions to algebraic representations of these functions. The easiest example is the result that every univariate real polynomial that is nonnegative as a function on $\mathbb{R}$ is a sum of squares of polynomials. Finding such algebraic certificates for nonnegativity is probably the most important question in the area of real algebraic geometry, and has also turned out to be very useful for polynomial optimization.

Usually one considers the $\mathbb{R}$-algebra $A$ of real polynomial functions on some real variety.  Then a subset $K$ of the real points of the variety is specified, and one examines the convex cone of all functions from $A$ that are nonnegative on $K$. The aim is to find an algebraic description of this cone, in terms of sums of squares and some polynomials used in the definition of $K$. One of the most important results is for example Schm\"udgen's Theorem: if $K$ is compact and defined by finitely many simultaneous polynomial inequalities, then every strictly positive polynomial function is obtained by addition and multiplication from sums of squares and the defining polynomials of $K$ \cite{S}. Results by Putinar \cite{P0} and Jacobi \cite{J} generalize this to representations that do not involve multiplication of the defining polynomials.
There are large amounts of further results in that direction, we suggest to the reader to consult \cite{KS, M} or \cite{pd} for detailed information.

In this work we consider a finitely generated algebra $A$ of real valued functions on some set $X$. Then to $A$ there corresponds a real variety $Y$, definable for example as the set of all real algebra homomorphisms on $A$. All the known Positivstellens\"atze can be applied, when considering elements of $A$ as functions on $Y$. But of course the more straightforward approach in this setup is to consider the elements of $A$ as functions on $X$ instead. Indeed $X$ is given directly with the definition of $A$, whereas $Y$ has to be computed first, and may not be obvious at all.  A nice Positivstellensatz should therefore give algebraic certificates for nonnegativity  on $X$, and not for the more hidden nonnegativity on $Y$.

An approach to this question can be found in the works of Lasserre and Putinar  \cite{LP,P}. A thorough analysis  shows that behind their results there is indeed a classical Positivstellensatz. The assumptions in their theorems always allow to obtain the classical positivity  on $Y$ from the assumed positivity on $X$. We will explain this in more detail below. We will then try to examine systematically, under which conditions the obvious positivity on $X$ implies the hidden positivity on $Y$. We can go beyond the results of Lasserre and Putinar, for example when considering non-continuous semialgebraic functions on  $X\subseteq\mathbb{R}^n$.

The formal setup is as follows. Let $A$ be a (unital) $\mathbb{R}$-algebra and let $Q$ be a quadratic module of $A$. We are mainly interested in the case where $A$ is finitely generated as an $\mathbb{R}$-algebra and the quadratic module $Q$ is archimedean and finitely generated as a quadratic module. Here  a \it quadratic module \rm of $A$ is a subset $Q$ of $A$ satisfying $Q+Q\subseteq Q$, $f^2Q\subseteq Q$ for all $f \in A$, and $1\in Q$. $Q$ is said to be \it archimedean \rm if for all $f \in A$ there is an integer $n\ge 1$ such that $n+f \in A$. $\sum A^2$ denotes the set of all finite sums of squares of elements of $A$. $\sum A^2$ is the unique smallest quadratic module of $A$. The quadratic module of $A$ generated by $g_1,\dots,g_s \in A$ is $$\operatorname{QM}(g_1,\dots,g_s) := \sum A^2+\sum A^2g_1+\cdots + \sum A^2 g_s.$$

We are interested in the case where $A$ is an algebra of real valued functions on some set $X$. Typically $X$ is a topological space and the elements of $A$ have some additional properties, e.g., they are continuous or borel measurable.

Actually, we don't need to assume that $A$ is an $\mathbb{R}$-subalgebra of $\mathbb{R}^X$, but only that we have an $\mathbb{R}$-algebra homomorphism $\bar{} : A \rightarrow \mathbb{R}^X$. Also there is no need to assume that $A$ separates points in $X$ although we can reduce to this case, if we want, by replacing $X$ by the set of equivalence classes of $X$ with respect to the equivalence relation $\sim$ on $X$ defined by $x_1 \sim x_2$ iff $\overline{f}(x_1)=\overline{f}(x_2)$ for all $f \in A$.
We denote by $K_{Q,X}$ the nonnegativity set of $Q$ in $X$, i.e., $$K_{Q,X} := \{ x \in X \mid  \bar{g}(x)\ge 0 \ \forall g \in Q \}.$$

Denote by $Y_A$ the set of all (unital) $\mathbb{R}$-algebra homomorphisms $y : A \rightarrow \mathbb{R}$. In case that only one algebra is involved, we just write $Y$ instead of $Y_A$. If we choose finitely many generators $h_1,\dots,h_t$ of $A$, then $Y$ embeds into $\mathbb{R}^t$, as the real points of the variety determined by $A$, i.e., $y \mapsto (y(h_1),\dots,y(h_t))$. We however prefer the coordinate free view. It also applies to algebras that are not finitely generated.

 We have an obvious map $m : X \rightarrow Y$ defined by $m(x)(f) = \bar{f}(x)$. $m(x)$ is naturally identified with the equivalence class of $x$ with respect to $\sim$. Each $f \in A$ defines a function $\hat{f} : Y \rightarrow \mathbb{R}$ given by $\hat{f}(y) := y(f)$. This coincides with the usual interpretation of $f$ as a polynomial on the variety of $A$. The map $\hat{} : A \rightarrow \mathbb{R}^Y$ is an $\mathbb{R}$-algebra homomorphism. $Y$ is given the weakest topology such that the functions $\hat{f}$, $f \in A$, are continuous.   In case that $Y$ is embedded as a variety into $\mathbb{R}^t,$ this is just the topology inherited from the euclidean topology. $K_{Q,Y}$ denotes the set of all $y\in Y$ such that $\hat{g}(y)\ge 0$ for all $g \in Q$.

 We are interested in understanding when, for all $f \in A$,

\medskip

P1: \ $\bar{f} >0 \text{ on } K_{Q,X} \ \Rightarrow \ f\in Q$,

\medskip
\noindent
or

\medskip

P2: \ $\bar{f} \ge 0 \text{ on } K_{Q,X}\ \Rightarrow \ f+\epsilon \in Q \ \forall \text{ real } \epsilon >0.$

\medskip
\noindent
One can consider other conditions as well, for example:

\medskip
P0: $\bar{f} \ge 0$ on $K_{Q,X}$ $\Rightarrow$ $f\in Q$,

\medskip
P3: $\bar{f} \ge 0$ on $K_{Q,X}$ $\Rightarrow$ $\exists$ $h\in A$ such that $f+\epsilon h \in Q$ $\forall$ real $\epsilon >0$,

\medskip
P4: $\bar{f} \ge 0$ on $K_{Q,X}$ $\Rightarrow$ $f \in Q^{\vee\vee}$ (the closure of $Q$ in the finest locally convex topology on $A$ \cite{CMN}), and

\medskip
P5: $\bar{f} \ge 0$ on $K_{Q,X}$ $\Rightarrow$ $\hat{f} \ge 0$ on $K_{Q,Y}$.

\medskip
\noindent
Obviously P0 $\Rightarrow$ P1 $\Rightarrow$ P2 $\Rightarrow$ P3 $\Rightarrow$ P4 $\Rightarrow$ P5.

\medskip

We record an immediate consequence of Jacobi's result in \cite{J}.

\begin{thm}\label{basic} Suppose the quadratic module $Q$ is archimedean.
\begin{itemize}
\item[(1)] If $m(K_{Q,X}) = K_{Q,Y}$ then P1 holds.
\item[(2)] If $\overline{m(K_{Q,X})} = K_{Q,Y}$ then P2 holds.
\end{itemize}
\end{thm}

Here $\overline{m(K_{Q,X})}$ denotes the closure of $m(K_{Q,X})$ in $Y$. Note that $m(K_{Q,X}) \subseteq K_{Q,Y}$ and $K_{Q,Y}$ is closed in $Y$ so $\overline{m(K_{Q,X})} \subseteq K_{Q,Y}$.

\begin{proof} We know by Jacobi's result, see \cite[Theorem 4]{J} or \cite[Theorem 2.3]{M0}, that $\hat{f} > 0$ on $K_{Q,Y}$ $\Rightarrow$ $f \in Q$. See \cite[Theorem 5.4.4]{M} for a simple proof. The result stated is immediate from this.
\end{proof}

\begin{rem} \label{rk} \

(1) If $Q$ is archimedean then $K_{Q,Y}$ is compact; see \cite[Theorem 5]{J} or \cite[Theorem 5.7.2(1)]{M}.

(2) If $\overline{m(K_{Q,X})} = K_{Q,Y}$ then P5 holds. Conversely, if P5 holds and either $K_{Q,Y}$ is compact or $A$ is finitely generated, then $\overline{m(K_{Q,X})} = K_{Q,Y}$.

\begin{proof} The first assertion is clear. For the second assertion, suppose first that $K_{Q,Y}$ is compact and $y\in K_{Q,Y}$, $y \notin \overline{m(K_{Q,X})}$.  By Stone-Weierstrass there is some $f \in A$, $\hat{f}>0$ on $\overline{m(K_{Q,X})}$, $\hat{f}(y)<0$, which contradicts P5. Suppose next that $A$ is finitely generated, say by $h_1,\dots,h_t$. Then, for any $y\in Y$ and any neighborhood $U$ of $y$ in $Y$, the element $f = \sum_{j=1}^t (h_j - y(h_j))^2-\epsilon$ is negative at $y$ and positive outside $U$, for $\epsilon >0$ sufficiently close to zero.
\end{proof}

(3) Combining items (1) and (2) with the fact that P2 $\Rightarrow$ P5, we deduce that the converse of Theorem \ref{basic}(2) is true, i.e., if $Q$ is archimedean and P2 holds then $\overline{m(K_{Q,X})} = K_{Q,Y}$. On the other hand, the converse of Theorem \ref{basic}(1) is not true. See Example \ref{ex1}(4) below for a counterexample.

(4) According to  \cite[Propositions 1.3, 2.1]{CMN}, $Q$ is archimedean iff $1$ is an algebraic interior point of $Q$ iff $Q$ has an algebraic interior point, and $f$ is an algebraic interior point of $Q$ iff $f$ is an interior point of $Q$ in the finest locally convex topology on $A$.

(5) In \cite[Proposition 1]{LP} and \cite[Theorem 2.1]{P} the hypothesis that  $m(K_{Q,X}) = K_{Q,Y}$ or $\overline{m(K_{Q,X})} = K_{Q,Y}$ is replaced by the hypothesis that $X$ is a topological space, $A$ is an $\mathbb{R}$-algebra of borel measurable functions on $X$, and $Q$ satisfies the \it moment property \rm relative to $X$, i.e., for all linear mappings $L : A \rightarrow \mathbb{R}$ such that $L\ge 0$ on $Q$ there is a positive borel measure $\mu$ on $X$, supported by $K_{Q,X}$, such that $L(f) = \int f d \mu$ $\forall$ $f\in A$. There is also an additional assumption that $X$ is a subspace of $\mathbb{R}^n$ for some $n\ge 1$.

(6) We claim the hypothesis that $Q$ satisfies the moment property relative to $X$ implies that $m(K_{Q,X}) = K_{Q,Y}$ or, at least, that $\overline{m(K_{Q,X})} = K_{Q,Y}$. Suppose first that $A$ is countably generated as an $\mathbb{R}$-algebra and $Q$ satisfies the moment property relative to $X$. We claim this forces $m(K_{Q,X}) = K_{Q,Y}$.

\begin{proof} Suppose $y \in K_{Q,Y}$, $y \notin m(K_{Q,X})$. Define $L$ by $L(f) = \hat{f}(y)$. Then $L \ge 0$ on $Q$ so we have a positive borel measure $\mu$ on $X$ such that $\int f d\mu =L(f) = \hat{f}(y)$ $\forall$ $f \in A$.  Fix a countable set of $\mathbb{R}$-algebra generators $\{ h_i\}$ for $A$. Replacing $h_i$ by $h_i-\hat{h_i}(y)$, we may assume $\hat{h_i}(y)=0$. Then $\int h_i^2 d\mu = 0$, so $\{ x \in X \mid h_i(x)\ne 0\}$ has $\mu$ measure zero, for each $i$. Since $y \notin m(K_{Q,X})$, there is no $x \in X$ satisfying $h_i(x)=0$ for all $i$. Thus $X$ is a countable union of sets of measure zero, so it has measure zero. Since $\mu(X) = \int 1 d\mu = 1$, this is a contradiction.
\end{proof}

Now drop the assumption that $A$ is countably generated, assuming only that $K_{Q,Y}$ is compact and $Q$ satisfies the moment property relative to $X$. We claim this forces $\overline{m(K_{Q,X})} = K_{Q,Y}$.

\begin{proof} Suppose $y \in K_{Q,Y}$, $y \notin \overline{m(K_{Q,X})}$. Fix a positive borel measure $\mu$ on $X$ such that $\int f d\mu = \hat{f}(y)$ $\forall$ $f \in A$. Since $K_{Q,Y}$ is compact, by Stone-Weierstrass, there exists $f\in A$, $\hat{f}>0$ on $\overline{m(K_{Q,X})}$, $\hat{f}(y)<0$. Then $f>0$ on $K_{Q,X}$, so $\int f d\mu >0$, which contradicts $\int f d\mu = \hat{f}(y)<0$.
\end{proof}

(7) In summary then, the hypothesis in Theorem \ref{basic} is easily understood and, at the same time, it is weaker than the corresponding hypothesis in \cite[Proposition 1]{LP} and \cite[Theorem 2.1]{P}.
\end{rem}

We see that the interesting question here is to identify cases where $m(K_{Q,X})=K_{Q,Y}$ or $\overline{m(K_{Q,X})}=K_{Q,Y}$ holds. Then the usual Positivstellens\"atze can be applied, but the nonnegativity needs to be  required only on $K_{Q,X}$ instead of $K_{Q,Y}$.   

The following examples are similar to the ones given in \cite{LP}.

\begin{exm} \label{ex1} \

(1) Let $X = \mathbb{R}$, $A$ the $\mathbb{R}$-algebra generated by $x(t)=\cos t$ and $y(t)=\sin t$, $Q = \sum A^2$. Then $A = \frac{\mathbb{R}[x,y]}{(x^2+y^2-1)}$, $Y = \{ (x,y) \in \mathbb{R}^2 \mid x^2+y^2=1\}$, $K_{Q,X} = X$, $K_{Q,Y} = Y$, $m(t) = (\cos t, \sin t)$, and $m(K_{Q,X})=K_{Q,Y}$. Thus Theorem \ref{basic} applies. Any $f \in A$ strictly positive on $\mathbb{R}$ is a sum of squares. Actually, P0 holds. Any $f\in A$ with $f\geq 0$ on $\mathbb{R}$ is a sum of squares. This follows from the corresponding well-known results for polynomials on the circle.

(2) Let $X = \mathbb{R}$, $A$ the $\mathbb{R}$-algebra generated by $x(t)=e^t$ and $y(t)=e^{-t}$, $Q = \operatorname{QM}(2-x^2,2-y^2)$. Then $A = \frac{\mathbb{R}[x,y]}{(xy-1)}$, $Y = \{ (x,y) \in \mathbb{R}^2 \mid xy=1\}$, $K_{Q,X} = [-\ln \sqrt2,\ln \sqrt2]$, $K_{Q,Y} = \{ (x,y) \mid xy=1,\, x^2\le 2,\, y^2\le 2\}$, $m(t) = (e^t, e^{-t})$. In this example, $m(K_{Q,X}) =\overline{m(K_{Q,X})}$ is properly contained in $K_{Q,Y}$, so Theorem \ref{basic} does not apply. On the other hand, if we let $Q' = \operatorname{QM}(2-x^2,2-y^2,x+y)$ then $K_{Q',X} = K_{Q,X}$ and $m(K_{Q',X}) = K_{Q',Y}$. By Theorem \ref{basic}, $f\in A$, $f> 0$ on $[-\ln \sqrt2, \ln \sqrt2]$ $\Rightarrow$ $f\in Q'$.

(3) Let $X= \mathbb{R} \backslash \{ 0 \}$, $A=$ the $\mathbb{R}$-algebra generated by $x(t)= t$ and $y(t)= \frac{|t|}{t}$, $Q = \operatorname{QM}(1-x^2)$. Then $A = \frac{\mathbb{R}[x,y]}{(y^2-1)}$, $Y = \{ (x,y) \in \mathbb{R}^2 \mid y^2 =1\}$, $K_{X,Q} = \{ t\in \mathbb{R} \backslash \{ 0 \} \mid -1 \le t \le 1\}\}$ and $K_{Q,Y} = \{ (x,y) \in \mathbb{R}^2 \mid y^2=1, -1\le x \le 1\}$. $\overline{m(K_{Q,X})}$ is properly contained in $K_{Q,Y}$. On the other hand, if we let $Q' = \operatorname{QM}(1-x^2, xy\}$ then $K_{Q',X} = K_{Q,X}$ and $\overline{m(K_{Q',X})} = K_{Q',Y}$, so Theorem \ref{basic} applies to $Q'$, even though it does not apply to $Q$. If $f \in A$ is $\ge 0$ on $K_{Q,X}$, then $f+\epsilon$ is expressible in the form $f+\epsilon = s_0+s_1(1-x^2)+s_2xy= s_0 + s_1(1-t^2)+s_2\vert t\vert$ for some $s_0,s_1,s_2 \in \sum A^2$.

(4) Let $X = \mathbb{R}$, $A =$ the $\mathbb{R}$-algebra generated by $x(t) = \frac{1}{1+t^2}$, $Q =
\operatorname{QM}(x,1-x)$. Then $A=\mathbb{R}[x]$, $Y = \mathbb{R}$, $m(t) = \frac{1}{1+t^2}$, $K_{Q,X} = \mathbb{R}$, $m(K_{Q,X}) = (0,1]$, and $K_{Q,Y} = [0,1]$. Property P0 holds: Any $f\in A$ satisfying $f\ge 0$ on $\mathbb{R}$ belongs to $Q$. This follows from the corresponding result on nonnegative polynomials on $[0,1]$.
On the other hand, since $m(K_{Q,X})$ is properly contained in $K_{Q,Y}$, $Q$ does not satisfy the moment property relative to $X$, by Remark \ref{rk} (6).
\end{exm}

One can consider other Positivstellens\"atze as well. For example, one can consider when, for all $f \in A$,

\medskip

Q1: \ $\bar{f} >0 \text{ on } K_{Q,X} \ \Rightarrow$ $\exists$ $p \in \sum A^2$ and $q \in Q$ such that $pf=1+q$,

\medskip
\noindent
or

\medskip

Q2: \ $\bar{f} \ge 0 \text{ on } K_{Q,X}\ \Rightarrow$ $\exists$ $p \in \sum A^2$, $q \in Q$ and $k\ge 0$ such that $pf=f^{2k}+q$.

\begin{thm} \label{basic2} Suppose the algebra $A$ is finitely generated and the quadratic module $Q$ is finitely generated and closed under multiplication. Then \begin{itemize}
\item[(1)] Q1 holds iff $m(K_{Q,X}) = K_{Q,Y}$.
\item[(2)] Q2 holds iff $\overline{m(K_{Q,X})} = K_{Q,Y}$.
\end{itemize}
\end{thm}

\begin{proof} ($\Leftarrow$): This is an immediate consequence of the Krivine-Stengle Positivstellensatz, see \cite[page 25]{M}. ($\Rightarrow$): Fix generators $h_1,\dots,h_t$ for $A$. If $y\in K_{Q,Y} \backslash m(K_{Q,X})$ then $f := \sum_{i=1}^t (h_i-y(h_i))^2$ is positive on $K_{Q,X}$ and zero at $y$, which contradicts Q1. If $y \in K_{Q,Y} \backslash \overline{m(K_{Q,X})}$ then, for $\epsilon >0$ sufficiently small, $f:= \sum_{i=1}^t (h_i-y(h_i))^2-\epsilon$ is positive on $K_{Q,X}$ and negative at $y$, which contradicts Q2.
\end{proof}

Theorem \ref{basic2} confirms once again the importance of understanding when the conditions $m(K_{Q,X}) = K_{Q,Y}$ and $\overline{m(K_{Q,X})} = K_{Q,Y}$ hold. According to \cite[Corollary 7.4.2]{M} the conclusion of Theorem \ref{basic2} continues to hold if the requirement that $Q$ is closed under multiplication is replaced by the requirement that $\dim(Y)\le 1$.

In Section 2 we consider when it is possible to enlarge the finitely generated quadratic module $Q$ of $A$ to a finitely generated quadratic module $Q'$ of $A$ which satisfies $K_{Q',X} = K_{Q,X}$, and $\overline{m(K_{Q',X})} = K_{Q',Y}$; see Corollary \ref{dim1} and Example \ref{dimge2}.
In Section 3, we start with an algebra $B$ of functions on $X$ and a quadratic module $Q$ of $B$ for which we already know that $m(X)=K_{Q,Y_B}$ or $\overline{m(X)}=K_{Q,Y_B}$ holds, and we consider ways in which $B$ and $Q$ can be extended to an algebra $A$ of functions on $X$ and a quadratic module $Q'$ of $A$ retaining these same properties; see Propositions \ref{inj}, \ref{alinj}, \ref{comp}, \ref{a} and \ref{b}. The results in Section 3, although not directly comparable to the results in \cite[Section 3]{LP}, are related to and were motivated by these results.
In Section 4 we give examples of the various sorts of extension described in Section 3.


\section{enlarging the quadratic module}

Suppose the $\mathbb{R}$-algebra $A\subseteq \mathbb{R}^X$ is finitely generated, say $A=\mathbb{R}[h_1,\dots,h_t]$, and the quadratic module $Q$ of $A$ is also finitely generated. We know $\overline{m(K_{Q,X})} \subseteq K_{Q,Y}$. As we have seen, special interest is attached to the case where $\overline{m(K_{Q,X})}= K_{Q,Y}$. Thus it is natural to wonder when there exists a finitely generated quadratic module $Q'$ of $A$ with $Q' \supseteq Q$, $K_{Q',X} = K_{Q,X}$, and $\overline{m(K_{Q',X})} = K_{Q',Y}$.

Note: If we drop the requirement that $Q'$ be finitely generated, then the existence of $Q'$ is more-or-less trivial. 
One can just adjoin to $Q$, for each point $y \in K_{Q,Y} \backslash \,\overline{m(K_{Q,X})}$, an element of the form $\sum_{j=1}^t (h_j-y(h_j))^2- \epsilon$, $\epsilon >0$ which is positive on $\overline{m(K_{Q,X})}$ and negative at $y$.

$Y$ is identified with the algebraic subset $\{ (y(h_1),\dots,y(h_t)) \mid y\in Y\}$ of $\mathbb{R}^t$. $K_{Q,Y} \subseteq Y$ is a basic closed semialgebraic set, i.e. defined by finitely many simultaneous polynomial inequalities. For simplicity, we restrict to the case where the set $\overline{m(K_{Q,X})}$ is semialgebraic too:

\begin{exm} Suppose $X$ is a semialgebraic subset of $\mathbb{R}^n$, for some $n\ge 1$, and $A$ is an algebra of semialgebraic functions on $X$.  The map $m: X \rightarrow Y$ is given by $\underline{x} \mapsto (h_1(\underline{x}),\dots, h_t(\underline{x}))$, which is a semialgebraic map, so $m(K_{Q,X})$ and $\overline{m(K_{Q,X})}$ are semialgebraic sets in $K_{Q,Y}$.
\end{exm}

Clearly, the quadratic module $Q'$ that we are looking for will exist iff the closed semialgebraic set $\overline{m(K_{Q,X})}$ is basic closed. Let $S := \overline{m(K_{Q,X})}$. According to Br\"ocker's Criterion, , e.g., see \cite[Theorem 7.1.1]{M-1}, $S$ is basic iff, for each real prime $\frak{p}$ of $A$, the associated constructible set $\widetilde{S}_{\frak{p}}$ in the real spectrum of the field of fractions of the domain $A/\frak{p}$ is basic. To check if $\widetilde{S}_{\frak{p}}$ is basic, one can use the so-called Fan Criterion; e.g., see \cite[Proposition 3]{Bro}, although this criterion is not always easy  to apply.

If $\dim(S) \le 1$ then $S$ is always basic; see \cite{Bro} or \cite[Theorem 7.5.4]{M-1}.

\begin{crl}\label{dim1}In the setup as above, if $\dim(m(K_{Q,X})) \le 1$ then there is a finitely generated quadratic module $Q'$ of $A$ with $Q' \supseteq Q$, $K_{Q',X} = K_{Q,X}$, and $\overline{m(K_{Q',X})} = K_{Q',Y}$.
\end{crl}

Note: One can reduce always to the case where $Y$ is the Zariski closure of $m(K_{Q,X})$. This just involves replacing $A$ by the $\mathbb{R}$-algebra $A/I$ where $I$ is the ideal of elements of $A$ vanishing on $K_{Q,X}$.

Already in the 2-dimensional case there are examples where no such $Q'$ (as in Corollary \ref{dim1}) exists:

\begin{exm}\label{dimge2} Let $X= \mathbb{R}^2$, $A = \mathbb{R}[u,v]$, where $u(x,y) = x$, $v(x,y) = |x|-|y|$, and let $Q$ be the quadratic module of $A$ generated by $1-u^2$ and $1-v^2$. $Y$ is identified with $\mathbb{R}^2$. $m : X \rightarrow Y$ is given by $(x,y) \mapsto (x,|x|-|y|)$. $K_{Q,Y}$ is $[-1,1]\times [-1,1]$. $m(K_{Q,X}) = \overline{m(K_{Q,X})} = \{ (u,v) \mid -1\le u\le 1, -1 \le v \le |u|\}$. This latter set is not basic, so there is no finitely generated $Q'$ such that  $Q' \supseteq Q$, $K_{Q',X} = K_{Q,X}$, and $\overline{m(K_{Q',X})} = K_{Q',Y}$. One shows that the set $S = m(K_{Q,X})$ is not basic by showing that $\widetilde{S}_{\frak{p}}$ is not basic,  where $\frak{p} = \{ 0 \}$. Since $A = \mathbb{R}[u,v]$ and $\frak{p} = \{ 0 \}$, the field of fractions of the domain $A/\frak{p}$ is the rational function field $\mathbb{R}(u,v)$. The Fan Criterion is violated. The orderings of the power series field $\mathbb{R}((u,v))$ compatible with the discrete valuation ring $\mathbb{R}[[u,v]]_{(u-v)}$ form a 4-element fan whose intersection with $\widetilde{S}_{\frak{p}}$ has 3 elements.
\end{exm}

\section{Enlarging the algebra}

In this section we start with some algebra $B$ of functions on $X$, and some quadratic module $Q$ of $B$, for which we already know that $m(X)=K_{Q,Y_B}$ or $\overline{m(X)}=K_{Q,Y_B}$ holds. We then want to enlarge the algebra and retain these equalities.

So let $B$ be an $\mathbb{R}$-subalgebra of $\mathbb{R}^X$, where $X$ is any set. Let $Q\subseteq B$ be a quadratic module with $K_{Q,X}=X.$ For any extension $B\subseteq A\subseteq \mathbb{R}^X$ of algebras we have a canonical continuous map $p\colon Y_A\rightarrow Y_B$, and if $Q'\subseteq A$ is a quadratic module with $Q'\cap B\supseteq Q$, then $p\colon K_{Q',Y_A}\rightarrow K_{Q,Y_B}$. If also $K_{Q',X}=X$ holds we have the following commutative diagram: $$\xymatrix{X \ar@{->}^{m\quad}[r] \ar@{->}_{m}[dr] &  \ar@{->}^{p}[d]   K_{Q',Y_A}  \\ & K_{Q,Y_B}  } $$

\begin{prop} \label{inj} Assume $m(X)=K_{Q,Y_B}$ and $A=B[f]$ for some $f\in A$. In each of the following cases the mapping $p\colon K_{Q',Y_A}\rightarrow K_{Q,Y_B}$ is injective, $K_{Q',X}=X$ holds,  and we thus have $m(X)=K_{Q',Y_A}$: \begin{itemize}
\item[(1)] $f=\root r\of g$, $r$ odd, $g\in B$, $Q'$ the quadratic module generated by $Q$ in $A$.
\item[(2)] $f=\root s\of g$, $s$ even, $g\in B$, $g\geq 0$ on $X$, and $Q'$ the quadratic module generated by $Q$ and $f$ in $A$.
\item[(3)] $f=\frac1g, $  $g\in B$, $g\neq 0$ on $X$, $Q'$ the quadratic module generated by $Q$ in $A$.
\item[(4)] $f= g\cdot \chi_{\{q\geq 0\}}+ h\cdot \chi_{\{q<0\}}$ for some $g,h,q\in B$, such that $q(x)=0$ implies $g(x)=h(x)$ for $x\in X$, and $Q'$ the quadratic module generated by $Q$ and $-q(f-g)^2, q(f-h)^2$ in $A$.\end{itemize}
\end{prop}

In (4), $\chi_S$ denotes the characteristic function of a set $S\subseteq X$. For the necessity of the assumptions in (4) see Example \ref{counter3} in the next section. Note that functions constructed from polynomials by applying (1),(2) and (3) finitely many times are also considered  in Section 3 of \cite{LP}.

\begin{proof}[Proof of Proposition \ref{inj}] First note that $K_{Q',X}=X$ in all of the cases. Then note that injectivity of $p$ implies $m(X)=K_{Q',Y_A}$. This can be seen from the above commutative diagram.

So let $y\in K_{Q',Y_A}$. In case (1) we find $y(f)^r=y(g)$, and thus $y(f)=y(g)^{\frac1r}$. So $y$ is already uniquely determined by its values on $B$, i.e. by $p(y)$. In case (2) we also have $y(f)^s=y(g)$, and since $y(f)\geq 0$, $y(f)$ is again uniquely determined by the values of $y$ on $B$. In case (3) we find $y(f)y(g)=1$. So $y(g)\neq 0$ and again $y$ is uniquely determined by $p(y)$.

In case (4) we get $y(f)= y(g)$ or $y(f)=y(h)$ from the identity $(f-g)(f-h)=0$ that holds in $A$. If $y(q)>0$ then $y(f)=y(g)$ since $-q(f-g)^2\in Q'$. If $y(q)<0$ then $y(f)=y(h)$ since $q(f-h)^2 \in Q'$. If $y(q)=0$ then $y(g)=y(h)=y(f)$, since $m(X)=K_{Q,Y_B}$ and the regularity assumption on $g,h,q$. All in all, $y$ is uniquely determined by $p(y)$ in all cases.
\end{proof}

Note that if $B$ consists of  continuous functions on $X$, then $A$ in the above Proposition can also only contain continuous functions.
If we want to adjoin something non-continuous, for example a characteristic function, we can not expect to get $m(X)=K_{Q',Y_A}$ for the bigger algebra:

\begin{exm}\label{char} Let $X=[-1,1]$, $B=\mathbb{R}[t]$ and $Q$ the quadratic module generated by $1-t^2$. We have $m(X)=K_{Q,Y_B}$. Let $f=\chi_{[0,1]}$ and $A=B[f]$. Consider the element $y\in Y_A$ defined by $y(a)= \lim\limits_{t\rightarrow 0^-} a(t)$. If $Q'$ is any quadratic module with $K_{Q',X}=X$, then $y\in K_{Q',Y_A}$. But one checks that $y$ does not belong to $m(X)$.
\end{exm}

However, if we only want to get $\overline{m(X)}=K_{Q',Y_A}$, we can go beyond the setup of continuous functions:

\begin{prop} \label{alinj} Assume $\overline{m(X)}=K_{Q,Y_B}$ and $A=B[f]$ for some $f\in A$. In each of the following cases we  also have $K_{Q',X}=X$ and  $\overline{m(X)}=K_{Q',Y_A}$: \begin{itemize}
\item[(1)] $(2)$ $(3)$ as in Proposition \ref{inj}.
\item[(4)] $f= g\cdot \chi_{\{q\geq 0\}}+ h\cdot \chi_{\{q<0\}}$ for some $g,h,q\in B$, such that $\hat{q}(y)=0$ implies $\hat{g}(y)=\hat{h}(y)$ for $y\in K_{Q,Y_B}$, and $Q'$ the quadratic module generated by $Q$ and $-q(f-g)^2, q(f-h)^2$ in $A$.
\item[(5)] $f=\chi_{\{q\geq 0\}}$ for some $q\in B$ fulfilling the following regularity condition: $ \{ y\in K_{Q,Y_B}\mid  \hat{q}(y)=0 \} \subseteq \overline{ \{ y\in K_{Q,Y_B} \mid \hat{q} (y)>0 \} } \cap   \overline{ \{ y\in K_{Q,Y_B} \mid \hat{q} (y)<0 \} },$ and $Q'$ the quadratic module generated by $Q$ and  $qf,q(f-1)$ in $A$.
\end{itemize}
\end{prop}

\begin{proof}
$K_{Q',X}=X$ is clear in all of the cases. Now let $y\in K_{Q',Y_A}$ and let $V$ be an open  neighborhood of $y$ in $K_{Q',Y_A}$. We have to show that $V$ contains some element from $m(X)$. By definition of the topology we can assume that $V$ is of the following form: $$V= \hat{f}^{-1}(U)\cap \bigcap_{i=1}^r \hat{b}_i^{-1}(U_i)\cap K_{Q',Y_A},$$ for some $b_i\in B$ and $U,U_i$ open subsets of $\mathbb{R}$. In case (1)  we consider the following open subset of $K_{Q,Y_B}:$ $$W:=\bigcap_{i=1}^r \hat{b}_i^{-1}(U_i)\cap \hat{g}^{-1}(U^r)\cap K_{Q,Y_B}.$$ $W$ is nonempty since $p(y)\in W$, and thus contains some point $m(x)$. But then also $m(x)\in V\subseteq K_{Q',Y_A}$, as one easily checks.  In case (2) we can assume that $U$ is closed under taking absolute values, using $y(f)\geq 0$. Then in the definition of $W$ we replace $\hat{g}^{-1}(U^r)$ by $\hat{g}^{-1}(U^s\cup -U^s)$ and repeat the above argument. In case (3) we can assume $0\notin U$, since $y(f)y(g)=1$. Then repeat the argument with $\hat{g}^{-1}(\frac{1}{U})$ in the definition of $W$. In case (4) first assume $y(q)>0$. Then $y(f)=y(g)$, and we consider $$W= \hat{q}^{-1}((0,\infty))\cap \hat{g}^{-1}(U)\cap\bigcap_i \hat{b}_i^{-1}(U_i)\cap K_{Q,Y_B}.$$ Is is nonempty and open, so contains some $m(x)$. From $q(x)>0$ we see $f(x)=g(x)\in U$, and thus $m(x)\in V$. The case $y(q)<0$ is similar. In case $y(q)=0$ use $$W=\hat{g}^{-1}(U)\cap\hat{h}^{-1}(U)\cap \bigcap_i \hat{b}_i^{-1}(U_i)\cap K_{Q,y_B}$$ and the assumption that $y(q)=0$ implies $y(g)=y(h)=y(f)$. In case (5) assume $y(q)>0$ and use $$W=\hat{q}^{-1}((0,\infty))\cap \bigcap_{i} \hat{b}_i^{-1}(U_i)\cap K_{Q,Y_B}.$$ For $m(x)\in W$ we find $f(x)=1$, and also $y(f)=1$ since $q(f-1)\in Q'$ and $f^2-f=0$. So $m(x)\in V$. $y(q)<0$ is similar. If now finally $y(q)=0$ and $y(f)=1$ say, then $$W=\hat{q}^{-1}(0,\infty))\cap \bigcap \hat{b}_i^{-1}(U_i)\cap K_{Q,Y_B}$$ is nonempty, using the regularity assumption on $\hat{q}$. Any point $m(x)\in W$ then also lies in $V$. The case $y(f)=0$ is similar.
\end{proof}

For the necessity of the regularity assumptions in (5), see Examples \ref{counter2} and \ref{counter4} below.
The following easy observation is also interesting:

\begin{lemma}\label{cont}
Let $X$ be a compact topological space and assume $A\subseteq\mathbb{R}^X$ contains only functions continuous on $X$. Then $\overline{m(X)}=m(X)$ in $Y_A$.
\end{lemma}
\begin{proof}
The mapping $m\colon X\rightarrow Y_A$ is continous. Since $X$ is compact and $Y_A$ is hausdorff, $m$ is a closed mapping. So $m(X)$ is closed in $Y_A$.
\end{proof}

 In some special cases we can relax the regularity condition from Proposition \ref{alinj} (5) to a condition involving the set $X$ only. This weaker condition then turns out to be necessary, however (see Example \ref{counter} in the next section).

\begin{prop} \label{comp} Assume $X$ is a compact topological space, $B$ contains only continuous functions on $X$ and  $\overline{m(X)}=K_{Q,Y_B}$. Suppose $q\in B$ fulfills the following regularity condition: $$\{x\in X\mid q(x)=0\}\subseteq\overline{\{x\in X\mid q(x)<0\}}.$$ Then for $f=\chi_{\{ q\geq 0 \}}$, $A=B[f]$, and $Q'$ the quadratic module generated by $Q$ and $qf,q(f-1)$ in $A$ we have $K_{Q',X}=X$ and $\overline{m(X)}=K_{Q',Y_A}.$
\end{prop}
\begin{proof}
$K_{Q',X}=X$ is clear. Again let $y\in K_{Q',Y_A}$. In case that $y(q)\neq 0$ a similar argument as above even shows that $y\in m(X)$, using Lemma \ref{cont}. In case $y(q)=0$ choose $x\in X$ with $p(y)=m(x)$ in $Y_B$. If $y(f)=1$ then $y=m(x)$. In case $y(f)=0$ take any open set $$V=\hat{f}^{-1}(U)\cap\bigcap_i\hat{b}_i^{-1}(U_i)\cap K_{Q',Y_A}$$ containing $y$.
Using the regularity assumption on $q$ take $z\in X$ with $q(z)<0$ and $b_i(z)\in U_i$ for all $i$. Then $f(z)=0$ and thus $m(z)\in V$.
\end{proof}

The above results now permit the following method. Start with a basic closed set $X=\{ x\in\mathbb{R}^n\mid g_1(x)\geq 0,\ldots,g_s(x)\geq 0\}$ for some $g_1,\ldots,g_s\in\mathbb{R}[x_1,\ldots,x_n]$. Let $A_0$ be the algebra of polynomial functions on $X$, and $Q_0$ the quadratic module generated by the functions $g_1,\ldots,g_s$ in $A_0$. Then clearly $m(X)=K_{Q_0,Y_{A_0}}$ holds. Now apply Propositions \ref{inj}, \ref{alinj} and \ref{comp} inductively, to produce a chain $$(A_0,Q_0)\subseteq \cdots \subseteq (A_k,Q_k)=(A,Q)$$ of algebras and quadratic modules, such that $m(X)=K_{Q_i,Y_{A_i}}$ or  $\overline{m(X)}=K_{Q_i,Y_{A_i}}$ holds in each step. Note that for checking the regularity assumptions from Proposition \ref{alinj} (4) and (5), it might be necessary to compute $K_{Q_i,Y_{A_i}}$ at several steps in the chain, whereas all other  conditions  can be checked on $X$ only.  In the resulting algebra $(A,Q)$ we can then for example apply Theorem \ref{basic}.  We  only have to take care about  $Q$ being archimedean:

\begin{prop}\label{a} Let $B$ be an algebra and $Q\subseteq B$ an archimedean quadratic module. Assume $A=B[f]$ for some $f\in A$ and $Q'\subseteq A$ a quadratic module with $Q\subseteq Q'\cap B$. In each of the following cases also  $Q'$ is archimedean in $A$:
\begin{itemize}
\item $f$ is integral over $B$
\item  $N-f^2\in Q'$ for some $N \in \mathbb{N}$
\end{itemize}
\end{prop}
\begin{proof} Denote by $H_{Q'}$ the set of all elements $a$ of $A$ such that $\ell \pm a \in Q'$ for some positive integer $\ell$. One knows that $H_{Q'}$ is an $\mathbb{R}$-algebra (\cite[Proposition 5.2.3]{M}) and $B\subseteq H_Q'$, since $Q'\cap B$ is archimedean in $B$. One also knows that $H_{Q'}$ is integrally closed in $A$\footnote{$H_{Q'}$ is even semi-integrally closed in $A$, i.e., if $a\in A$ and $\exists$ $n\ge 1$ and $a_i \in H_{Q'}$, $i=0,\dots, 2n-1$ such that $-a^{2n}+\sum_{i=0}^{2n-1} a_ix^i \in Q'$, then $x\in H_{Q'}$.  This can be deduced, for example, from \cite[Proposition 6.3.1]{B}.}.  So if $f$ is integral over $B$, $f$ is contained in $H_{Q'}$. If $N-f^2\in Q'$, then also $f\in H_{Q'}$. In both cases $A=H_{Q'}$, which means that $Q'$ is archimedean in $A$.
\end{proof}	

So in our inductive construction of an algebra $(A,Q)$, assume that $X\subseteq\mathbb{R}^n$ is compact. Then $N-\sum x_i^2\geq 0$ on $X$ for some big enough $N$. If we include such $N-\sum_{i=1}^n x_i^2$ to $Q_0$, then $Q_0$ is archimedean in $A_0$ and still $K_{Q_0,X}=X$ holds.  In any step where we adjoin to $A_i$ an element $f$ of the form (1), (2), (4) or (5), $f$ is integral over $A_i$. Indeed we have $f^r\in A_i, f^s\in A_i, (f-g)(f-h)=0$ and $f^2-f=0$ in the respective cases. So $Q_{i+1}$ remains archimedean, if $Q_i$ was.

 In case (3) where we adjoin some $f=\frac1g$ with $g\in A_i$ this is not necessarily true. We have to include some $N- f^2$ to $Q_{i+1}$. Note however that this might not be possible without destroying $K_{Q_{i+1},X}=X$. In Example \ref{char} look at $g:=(1-\chi_{[0,1]})t^2+\chi_{[0,1]}$. The function $g$ is strictly positive everywhere on $X$, but takes values arbitrary close to zero. So $f=\frac1g$ is not bounded on $X$.  If $g$ is continuous however, this problem does not occur.

Finally note that we  can of course start the construction process with any algebra $(A_0,Q_0)$ that fulfills $m(X)=K_{Q_0,Y_{A_0}}$. Taking the algebra of polynomial functions on a basic closed semialgebraic set $X$ is just the most convenient starting point.

We finish this section with a converse to the above extension theorems.

\begin{prop}\label{b} Let $B\subseteq A$ be an extension of $\mathbb{R}$-algebras, let $Q'\subseteq A$ be an archimedean quadratic module and set $Q:= Q'\cap B$. Then $p\colon K_{Q',Y_A}\rightarrow K_{Q,Y_B}$ is onto. Thus if $B\subseteq A\subseteq\mathbb{R}^X$, then $m(X)=K_{Q',Y_A}$ implies $m(X)=K_{Q,Y_B}$, and $\overline{m(X)}= K_{Q',Y_A}$ implies $\overline{m(X)}=K_{Q,Y_B}.$

\end{prop}
\begin{proof}
$Q'$ is archimedean in $A$, and thus $Q$ is archimedean in $B$. So both $K_{Q',Y_A}$ and $K_{Q,Y_B}$ are compact. So $p(K_{Q',Y_A})$ is compact and therefore closed in $K_{Q,Y_B}$. Now assume there is some $y\in K_{Q,Y_B},$ $y\notin p(K_{Q',Y_A})$. By Stone-Weierstrass there is some $b\in B$ with $\hat{b}(y)<0$ and $\hat{b}>0$ on $p(K_{Q',Y_A})$. But then $b$, considered as an element from $A$, is positive on $K_{Q',Y_A}$, and thus belongs to $Q'$, by Jacobi's Theorem (\cite[Theorem 4]{J}). But then also $b\in Q$, which contradicts $\hat{b}(y)<0$. So we have shown that $p$ is onto. The rest of the claim is now clear from the above commutative diagram.
\end{proof}

\section{Examples}

\subsection{} Consider $A=\mathbb{R}[t, \root 3 \of t, f],$ with $$f(t)=\left\{ \begin{array}{ccc} t^2 & \mbox{ for } & t\leq 0 \\ \root 3 \of t & \mbox{ for } & t\geq 0.\end{array}\right. $$  We get $m(\mathbb{R})= K_{Q,Y_A}$, if we let $Q$ be the quadratic module of $A$ generated by $-t(f- \root 3 \of t)^2$ and $t(f-t^2)^2$. Thus for any quadratic module $Q'\supseteq Q$ also $m(K_{Q',\mathbb{R}})=K_{Q',Y_A}$. For example, if we take $Q'$ to be generated by $Q$ and $1-t^2$, then $K_{Q',\mathbb{R}}=[-1,1]=:X$. $Q'$ is archimedean. So every function from $A$ that is strictly positive on $X$ belongs to $Q'$.

\subsection{} Let $X=[-1,1]$ and $A=\mathbb{R}[t,\vert t\vert, \chi_{[0,1]}]$. We start with $A_0=\mathbb{R}[t]$ and $Q_0$ the quadratic module generated by $1-t^2$, which is archimedean. Then we adjoin $\vert t \vert = \root 2 \of{ t^2}$ to obtain $A_1$. We have to add $\vert t\vert$ as a generator of $Q_1$. $Q_1$ is again archimedean. Finally we adjoin $\chi_{[0,1]}$ to obtain $A$. We can use Proposition \ref{comp} with $q= t$. We have to add to $Q_1$ the generators $t\chi_{[0,1]}$ and $t(\chi_{[0,1]}-1)$. So $Q$, generated by $1-t^2,\vert t\vert, t\chi_{[0,1]}$ and $-t\chi_{[-1,0)}$ in $A$, is archimedean with $\overline{m(X)}=K_{Q,Y_A}$. For each function $a$ from $A$ that is nonnegative on $[-1,1]$ we have $a+\epsilon \in Q$, for all $\epsilon >0$.

\subsection{} In this example we start with $A_0=\mathbb{R}[\sin t, \cos t],$ $X=\mathbb{R}$ and $Q_0=\sum A_0^2$. As we have seen in  Example \ref{ex1}, $m(X)=K_{Q_0,Y_{A_0}}$ holds. In the next step we adjoin $\chi_S$, where $S=\{x\in X\mid \cos(x)\geq 0\}$. We cannot apply Proposition \ref{comp}, since $X$ is not compact. So we have to check the reqularity condition in Proposition \ref{alinj} (5) for $q=\cos t=y$ on $K_{Q_0,Y_{A_0}}=m(X)= S^1\subseteq\mathbb{R}^2$. This is easily done. So let $Q$ be the quadratic module in $A=\mathbb{R}[\sin t,\cos t, \chi_S]$ generated by $\chi_S \cdot\cos t$ and $(\chi_S -1)\cdot\cos t$. We have $K_{Q,X}=X$ and $\overline{m(X)}=K_{Q,Y_A}$.

\subsection{}\label{counter} We give an example to justify the regularity assumption on $q$ in Proposition \ref{comp}. Therefore let $X=[-2,2]$ with the usual topology, $B=\mathbb{R}[t]$ and $Q$ generated by $2-t,2+t$.  $X$ is compact, all functions from $B$ are continuous, and $m(X)=K_{Q,Y_B}$ holds. Now take $q=-t^2(t+1)(t-1)$. $q$ vanishes at $0$, which is not in the closure of the points where $q$ is negative. We anyhow adjoin $f=\chi_{\{q\geq 0\}}=\chi_{[-1,1]}$ to $B$ an add the generators $qf, q(f-1)$ to $Q$ to obtain $Q'\subseteq A=B[f]$. One computes $Y_A= \mathbb{R}\times\{0,1\}\subseteq\mathbb{R}^2$ and $$K_{Q',Y_A}= \left([-2,-1]\times\{0\}\right) \cup \left([-1,1]\times\{1\}\right) \cup \left([1,2]\times\{0\}\right)  \cup \{(0,0)\}.$$ The isolated point $(0,0)$ of $K_{Q',Y_A}$ reflects precisely the zero of $q$ at $0$. Now one easily checks that $\overline{m(X)}$ equals $K_{Q',Y_A}$ \textit{without} this isolated point. So the regularity assumption on $q$ in Proposition \ref{comp} can not be dropped.

Of course one can add further generators to $Q'$ to solve the problem. For example, if one adds $t^2 + f^2 -\frac12$, then still $K_{Q',X}=X$, and now the point $\{(0,0)\}$ is removed from $K_{Q',Y_A}$. Even simpler would be to start with a different polynomial $q$. For example, $q=1-t^2$ would do.

\subsection{}\label{counter3} A similar example shows that the regularity condition on $q,g,h$ from Proposition \ref{inj} (4) is necessary, and it is not enough to just assume that the piecewisely defined function $f$ is continuous. Indeed consider $X=[-2,2]$, $B=\mathbb{R}[t]$ and $Q$ generated by $2\pm t$. We then take $q=-t^2(t+1)(t-1)$ as above, and $g=1-t^2, h=t^2-1$. Then $f:= g\cdot\chi_{\{ q\geq 0\}} + h\cdot \chi_{\{q<0\}}$ is continuous, although $g$ and $h$ do not coincide at each zero of $q$ in $X$. Let $A=B[f]$ and $Q'$ generated in $A$ by $Q$ and $-q(f-g)^2, q(f-h)^2$.   We leave it to the reader to check that there is again an isolated point in $K_{Q',Y_A}$ that does not belong to $m(X)$ or its closure.

Again the problem arises from taking the wrong polynomial $q$. If we take $q=1-t^2$ in this example, everything works fine.

\subsection{}\label{counter2} We discuss the regularity assumption of  Proposition \ref{alinj} (5). In contrast to Proposition \ref{comp}, it involves the topological space $K_{Q,Y_B}$ instead of $X$ only. But this is clear since $X$ is not assumed to carry any topology. But even if $X$ is a subset of $\mathbb{R}^n$, the result fails if the regularity of the zero set of $q$ is  only fulfilled on $X$ instead of on $K_{Q,Y_B}$; see Example \ref{counter4} below.

One might also wonder why the regularity condition in Proposition \ref{alinj} (5) involves the set where $q$ takes positive values, and not only the set where $q$ is negative, as in Proposition \ref{comp}. We show that this is necessary. Therefore let  $X=[-2,2]\subseteq\mathbb{R}$ and $B=\mathbb{R}[t,\chi_S], $ where $S=[0,2]$. We equip $B$ with the quadratic module generated by $2\pm t,   t\cdot\chi_S, t\cdot(\chi_S-1)$. We find $K_{Q,X}=X$ and  $\overline{m(X)}=K_{Q,Y_B}=\left([-2,0]\times\{0\}\right) \cup \left([0,2]\times\{1\}\right).$

Now let $q=-t^2(t+1)$. We have $T:=\{q\geq 0\}=[-2,-1]\cup\{0\}, $ and we want to adjoin $\chi_T$.  Note that $\hat{q}$ vanishes at three points in $K_{Q,Y_B}$, namely $(-1,0),(0,0)$ and $(0,1)$. All three points can be approximated by points where $\hat{q}$ is negative, but $\hat{q}$ does not take positive values around $(0,0)$ and $(0,1)$.

Let $A=B[\chi_T]$ and $Q'$ the quadratic module generated by $Q, q\cdot\chi_T$ and  $q\cdot(\chi_T-1)$. We have $A\cong\mathbb{R}[x,y,z]$ with the following relations: $$y^2=y, z^2=z, xyz=0.$$ So $Y_A$ consists of three distinct copies of $\mathbb{R},$ one copy for each of the cases $y=0$ and $z=0$; $y=0$ and $z=1$; $y=1$ and $z=0$, and an additional point for the case $y=1$ and $z=1$. $K_{Q',Y_A}$ is the union of the following sets, corresponding to the different cases: \begin{align*} y=0,z=0: & \quad [-1,0] \\ y=0, z=1: & \quad [-2,-1]\cup\{0\} \\ y=1,z=0: &\quad [0,2] \\ y=1,z=1: & \quad \{0\}  \end{align*} Now $m(X)$ is easily computed, again case by case:  \begin{align*} y=0,z=0: & \quad (-1,0) \\ y=0, z=1: & \quad [-2,-1] \\ y=1,z=0: &\quad (0,2] \\ y=1,z=1: & \quad \{0\}  \end{align*}
So we see that the closure of $m(X)$ does not contain the point $\{0\}$ that occurs in the case $y=0,z=1$. So the regularity assumption in Proposition \ref{alinj} (5) on $\hat{q}$ is really necessary, and can in general not be replaced by the same assumption involving the negativity set of $\hat{q}$ only.

\subsection{}\label{counter4} One can ask why the regularity assumption in Proposition \ref{alinj} (5) involves the set $K_{Q,Y_B}$, and not $X$ only. But even if $X\subseteq\mathbb{R}^n$ is given the usual topology, the regularity of $q$ on $X$  instead of $K_{Q,Y_B}$ is not enough for the result. If we consider the same example as in \ref{counter2}, but now take $q=-t(t+1)$ instead of $q=-t^2(t+1)$ (so $T=[-1,0]$), then all zeros of $q$ in $X$ can be approximated both by points where $q$ is negative and positive. However, on $K_{Q,Y_B}$ this is not true, as one easily checks.  And a similar computation as before shows that indeed $\overline{m(X)}\neq K_{Q',Y_A}$.

\end{document}